\documentclass[a4paper,12pt]{amsproc}

\title{A number of topologies on $C(X)$ lying between point-open topology and the topology of unifrom convergence}
\usepackage{amsfonts, amsmath, amssymb, graphicx, mathrsfs, geometry}
\usepackage[all]{xy}
\usepackage{tikz,float}
\usetikzlibrary{shapes.geometric, arrows}
\tikzstyle{arrow} = [thick,->,>=stealth]

\usetikzlibrary{cd, arrows.meta, positioning}
\usepackage{enumitem}
\newdir{ >}{!/6pt/@{ }*:(1,0.2)@_{>}*:(1,-0.2)@^{>}}

\theoremstyle{plain}

\newtheorem{theorem}{Theorem}[section]

\title{A number of properties enjoyed by two specially constructed  topologies on $C(X)$}
\usepackage{amsfonts, amsmath, amssymb, graphicx, mathrsfs}
\usepackage[all]{xy}

\newdir{ >}{!/6pt/@{ }*:(1,0.2)@_{>}*:(1,-0.2)@^{>}}

\theoremstyle{plain}

\newtheorem{lemma}[theorem]{Lemma}

\theoremstyle{definition}
\newtheorem{definition}[theorem]{Definition}

\newtheorem{remark}[theorem]{Remark}
\newtheorem{remarks}[theorem]{Remarks}
\newtheorem{counter example}[theorem]{Counter Example}
\newtheorem{notation}[theorem]{Notation}

\newtheorem{corollary}[theorem]{Corollary}

\newtheorem{example}[theorem]{Example}

\numberwithin{equation}{section}

\author[S. Dey]{Soumajit Dey}	\address{Department of Pure Mathematics, University of Calcutta, 35, Ballygunge Circular Road, Kolkata 700019, West Bengal, India}	\email{deysoumajit8@gmail.com}

\author[S. K. Acharyya]{Sudip Kumar Acharyya}	\address{Department of Pure Mathematics, University of Calcutta, 35, Ballygunge Circular Road, Kolkata
	700019, West Bengal, India}	\email{sdpacharyya@gmail.com}

\author[D. Mandal]{Dhananjoy Mandal} \address{Department of Pure Mathematics, University of Calcutta, 35 , Ballygunge Circular Road, Kolkata 700019, West Bengal, India}  \email{dmandal.cu@gmail.com / dmpm@caluniv.ac.in}
\keywords{$\aleph_\alpha$-$separable$; $Pathwise connected$; $q$-space; \v{C}ech-complete; cellularity; weight; Lindel\"{o}f number}
\subjclass[2020]{ 54A25, 54C30, 54C35, 54D45}
\begin{document}
	\title [The space $C_{\aleph_\alpha}(X)$]{A special kind of topology on $C(X)$ lying between the point-open topology and the topology of uniform convergence}
	\thanks {The first author extends immense gratitude and thanks to the University Grants Commission, New Delhi, for the award of research fellowship (NTA Ref. No. 211610214962).}
	\maketitle
	
	\begin{abstract}
		Let $X$ be a topological space. For each transfinite cardinal number $\aleph_\alpha$, we define a topology $C_{\aleph_\alpha}(X)$ on the ring $C(X)$. With $\aleph_\alpha=\aleph_0$, $C_{\aleph_\alpha}(X)$ reduces to the space $C_p(X)$. We prove that for $\aleph_\alpha\geq \aleph_1$, $C_{\aleph_\alpha}(X)$ is pathwise connected if and only if it is connected if and only if $X$ is pseudocompact. Here, we define $\aleph_\alpha$-separable space and furthermore, we show that $X$ is $\aleph_\alpha$-separable when and only when $C_{\aleph_\alpha}(X)$ is metrizable when and only when it is a sequential space. Later we introduce two new cardinal functions, namely $cc_{\aleph_\alpha}(X)$ and $ac_{\aleph_\alpha}(X)$ which turned out to be the character and pseudocharacter of $C_{\aleph_\alpha}(X)$ respectively. At the end of this article we show that a number cardinal functions associated with the space $C_{\aleph_\alpha}(X)$ are equal.
	\end{abstract}

	\bibliographystyle{plain}
	\section{Introduction}
	
	Given a Tychonoff topological space $X$, there are, in the literature, a variety of topologies on the ring $C(X)$ of all real-valued continuous functions defined on $X$. The point-open topology on $C(X)$ usually denoted by $C_p(X)$ and the topology of uniform convergence on $C(X)$, familiarly designated as $C_u(X)$ are two such prominent topologies. There is a big literature addressing various problems concerning these two topologies. In the present paper we deal with a number of pertinent problems associated with yet another topology that lies between these two topologies on $C(X)$. To be specific corresponding to transfinite initial ordinal number $\aleph_\alpha$, let $\mathcal{A}_{\aleph_\alpha}=\{A\subseteq X:|A|<\aleph_\alpha\}$. For any $f\in C(X)$, $\epsilon>0$ and $A\in \mathcal{A}_{\aleph_\alpha}$, let $$B(f,A,\epsilon)=\{g\in C(X):\sup\limits_{x\in A}|f(x)-g(x)|<\epsilon\}.$$	Then the family $\mathcal{B}_{\aleph_\alpha}=\{B(f,A,\epsilon):f\in C(X),A\in \mathcal{A}_{\aleph_\alpha}\text{ and }\epsilon>0\text{ a real number}\}$ is an open base for a topology on $C(X)$, which we call the $\aleph_\alpha$-topology on $C(X)$ and denote it by $C_{\aleph_\alpha}(X)$. Incidentally several interesting problems connected with a special version of this topology viz., $\aleph_1$-topology is already investigated in \cite{CRMR 2025,PM 2013,RS 2021}. We would like to mention that the space $C_s(X)$, initiated in \cite{PM 2013} is the same as $C_{\aleph_1}(X)$. It turns out that a few results concerning $C_s(X)$ in the articles \cite{CRMR 2025,RS 2021} are special cases of facts, obtained in our present paper. However there is a little overlap between the results achieved in our paper on using different techniques and these established in \cite{CRMR 2025,PM 2013,RS 2021}. To substantiate these comments, it is high time we narrate briefly about the organization of the technical section of our article.
	
	In Section 2 of this article we first realize that given an initial ordinal number $\aleph_\alpha$ and a topological space $X$ (not necessarily Tychonoff), it is possible to construct a Tychonoff space $Y$ such that the rings $C_{\aleph_\alpha}(X)$ and $C_{\aleph_\alpha}(Y)$ are topologically isomorphic. In the same section we cite several examples to illustrate the independence of the topology of the space $C_{\aleph_\alpha}(X)$ with a few other known topologies viz., the compact-open topology on $C(X)$, compact-$G_\delta$-open topology on $C(X)$, pseudo-compact open topology on $C(X)$ each of which lies between the point-open topology and the uniform topology on $C(X)$.
	
	On of the main results proved in Section 3 says that with the restriction $\aleph_\alpha\geq \aleph_1$, $C^*(X)$ of all bounded real-valued continuous functions on $X$ is a pathwise connected subspace of $C_{\aleph_\alpha}(X)$ and in particular a closed subset of $C_{\aleph_\alpha}(X)$. This clearly extends the fact that $C^*(X)$ is dense in $C_s(X)$ if and only if $X$ is pseudocompact, a result proved in Proposition 5.2 in \cite{PM 2013}. It is also realized in the same section that the component, the quasicomponent and the path component of any $f\in C_{\aleph_\alpha}(X)$ are all identical and equal to $f+C^*(X)$. It is easy to check that $C_{\aleph_\alpha}(X)$ is a topological group. It turns out that $C_{\aleph_\alpha}(X)$ is a topological ring if and only if $X$ is pseudocompact [compare with Corollary 3.11 in \cite{PM 2013}]. We also realize that the number of non-empty clopen sets in $C_{\aleph_\alpha}(X)$ is either 1 or or at least $\mathfrak{c}$. We also prove in the same section that several localized version of compactness like conditions for the space  $C_{\aleph_\alpha}(X)$ coincide. It is established that $C_{\aleph_\alpha}(X)$ is locally compact if and only if it is locally pseudocompact if and only if $X$ is finite.

	In Section 4, we find out several conditions involving the space $C_{\aleph_\alpha}(X)(\neq (C_p(X)))$, each necessary and sufficient for the metrizability of $C_{\aleph_\alpha}(X)$. Thus for instance it is realized that $C_{\aleph_\alpha}(X)(\neq (C_p(X)))$ is metrizable by a complete metric if and only if $C_{\aleph_\alpha}(X)=C_u(X)$. Any one of these condition holds good if and only if $X$ is $\aleph_\alpha$-separable (defined in an appropriate place) (Theorem \ref{metrizability}) which is further equivalent to the \v{C}ech-completeness of $C_{\aleph_\alpha}(X)$ (Theorem \ref{completeness Calpha}). In particular $C_s(X)=C_{\aleph_1}(X)$ is metrizable if and only if $X$ is separable if and only if $C_s(X)$ is \v{C}ech-complete \cite[Proposition 4.1]{CRMR 2025}. 
	
	In Section 5, we establish a number of facts demonstrating interrelation between various cardinal functions associated with the space $C_{\aleph_\alpha}(X)$. In this section, we introduce two new cardinal functions, viz., the almost-$\aleph_\alpha$-number $ac_{\aleph_\alpha}(X)$ and the $\aleph_\alpha$-closed number $cc_{\aleph_\alpha}(X)$ for the space $X$. We realize that the pseudocharacter of $C_{\aleph_\alpha}(X)$ is equal to $ac_{\aleph_\alpha}(X)$ (Theorem \ref{pseudocharacter Calpha}). It follows as a consequence, unlike the compact-open topology that $C_{\aleph_\alpha}(X)((\neq (C_p(X))) $ is metrizable if and only if it is submetrizable. Furthermore, we show that the character of the space $C_{\aleph_\alpha}(X)$ is identical to $cc_{\aleph_\alpha}(X)$ (Theorem \ref{character Calpha}). In the concluding theorem of Section 5, we show that the conditions separability and second countability of the space $C_{\aleph_\alpha}(X)$ become equivalent and equivalent to the compactness and metrizability condition for the ambient space $X$ (Theorem \ref{all equivCalpha}).

	\section{The space $C_{\aleph_\alpha}(X)$ and the adequacy of Tychonoff spaces $X$ to study $C_{\aleph_\alpha}(X)$ }
	
		\begin{theorem}\label{tychonoff calpha}
		Given a topological space $X$ and a transfinite cardinal number $\aleph_\alpha$, it is possible to construct a Tychonoff space $Y$ such that the spaces $C_{\aleph_\alpha}(X)$ and $C_{\aleph_\alpha}(Y)$ are topologically isomorphic.
	\end{theorem}
	
	\begin{proof}
		We define a binary relation `$\sim$' on $X$ by the following rule: for $a,b\in X$, $a\sim b$ if and only if $f(a)=f(b)$ for each $f\in C(X)$. This `$\sim$' is an equivalence relation on $X$. Let $Y=X/\sim\equiv$ the set of all the corresponding `$\sim$' equivalent disjoint classes. Let $\tau:X\rightarrow Y$ be the canonical map $\tau(x)=[x]$$\equiv$ the equivalence class in $Y$ which contains $x$. For each $f\in C(X)$, we can associate a unique function $g_f:Y\rightarrow \mathbb{R}$ by the rule: $g_f\circ \tau=f$. We equip $Y$ with the smallest topology with respect to which each $g_f$, $f\in C(X)$ is a continuous map on $Y$. Then $\tau:X\rightarrow Y$ is a continuous map, $Y$ is a Tychonoff space and $\psi: C(Y)\rightarrow C(X)$ defined by $\psi(g)=g\circ \tau$ is a lattice isomorphism from the ring $C(Y)$ onto the ring $C(X)$ (for a detailed proof of the assertion, see \cite[Chapter 3, Theorem 3.9]{GJ 1960}). To complete this theorem, we need to show that $\psi$ is a homeomorphism between the spaces $C_{\aleph_\alpha}(Y)$ and $C_{\aleph_\alpha}(X)$. Since $C_{\aleph_\alpha}(Y)$ and $C_{\aleph_\alpha}(X)$ are topological groups, therefore homogeneous spaces, it suffices to show that $\psi$ exchanges the basic open neighborhoods of $\underline{0}$ in their respective spaces. For a subset $A$ of $Y$ with $|A|< \aleph_\alpha$ and $\epsilon>0$, we set $B_Y(\underline{0},A,\epsilon)= \{g\in C(Y):|g(y)|\leq\epsilon\text{ for all }y\in A\text{}\}$. Then the family of all such sets $B_Y(\underline{0},A,\epsilon)$, with $A\subseteq Y$ with $|A|<\aleph_\alpha$ and $\epsilon>0$ is a base for the neighborhoods of $\underline{0}$ in the space $C_{\aleph_\alpha}(Y)$. We first show that for each such choice of $A$ and $\epsilon>0$, $\psi(B_Y(\underline{0},A,\epsilon)=B_X(0,G,\epsilon)\equiv\{f\in C(X):|f(x)|\leq\epsilon\text{ for all }x\in G\}$ for some subset $G$ of $X$ with $|G|<\aleph_\alpha$. Indeed, for each $a\in A$ we can choose a point $c\in X$ such that $\tau(c)=a$, because $\tau:X\rightarrow Y $ is an onto map. Let $G$ be the aggregate of all such points $c\in X$ then $|G|\leq |A|<\aleph_\alpha$. We claim that $\psi(B_Y(\underline{0},A,\epsilon))=B_X(\underline{0},G,\epsilon)$. \underline{Proof of this claim:} let $g\in B_Y(\underline{0},A,\epsilon)$, then $g(y)\in [-\epsilon,\epsilon]$ for each $y\in A$, this implies that for each $x\in G$, $g(\tau(x))\in [-\epsilon,\epsilon]$ or $\psi(g(x))\in [-\epsilon,\epsilon]\implies\psi(g)\in B_X(\underline{0},G,\epsilon)$. Thus, $\psi(B_Y(\underline{0},A,\epsilon))\subseteq B_X(\underline{0},G,\epsilon)$. On the other hand if $f\in C(X)$ in such that $f\in B_X(\underline{0},G,\epsilon)$, then $|f(x)|\leq \epsilon$ for each $x\in G$. Now since $f=\psi(g)=g\circ \tau$ for some $g\in C(Y)$, this implies that $|g(\tau(x))|\leq\epsilon$ for each $x\in G$ which implies that $|g(y)|\leq\epsilon$ for each $y\in A$, i.e., $g\in B_Y(\underline{0},A,\epsilon)$. So, $B_X(\underline{0},G,\epsilon)\subseteq \psi(B_Y(\underline{0},A,\epsilon))$. Thus $\psi$ carries basic neighborhood of $\underline{0}$ in the space $C_{\aleph_\alpha}(Y)$ to basic neighborhood of $\underline{0}$ in the space $C_{\aleph_\alpha}(X)$. By making analogous arguments, it is not hard to prove that $\psi$ pulls back basic neighborhood of $\underline{0}$ in the space $C_{\aleph_\alpha}(Y)$ to the space $C_{\aleph_\alpha}(X)$. Therefore, altogether $\psi:C_{\aleph_\alpha}(Y)\rightarrow C_{\aleph_\alpha}(X)$ is a homeomorphism. 
	\end{proof}

	   In this section, our primary objective is to compare $C_{\aleph_\alpha}(X)$ with the various other topologies on $C(X)$, for a transfinite cardinal number $\aleph_\alpha$. It is easily verified that $\aleph_\alpha\leq \aleph_\beta$ implies $C_{\aleph_\alpha}(X)\subseteq C_{\aleph_\beta}(X)$. Furthermore, it is not hard to check that $C_p(X)=C_{\aleph_0}(X)\subseteq C_s(X)=C_{\aleph_1}(X)\subseteq C_{\aleph_2}(X)\subseteq\cdots\subseteq C_u(X)$ holds; moreover,  $C_p(X)=C_{\aleph_\alpha}(X)$ if and only if either $X$ is finite or $\aleph_\alpha=\aleph_0$. At this point, it is important to mention that there are few other topologies lying between $C_p(X)$ and $C_u(X)$ other than $C_k(X)$, namely the compact-$G_\delta$-open topology or $C_{kz}(X)$ and the pseudocompact open topology or $C_{ps}(X)$. Let us consider the following sets.
	\begin{itemize}

	\item   $\mathcal{A}_k=\{A\subseteq X:A\text{ is a compact subset of }X\},$
	
	 \item  $\mathcal{A}_{kz}=\{A\subseteq X:A\text{ is a compact-}G_\delta\text{-set in }X\},$ 
	 \item  $\mathcal{A}_{ps}=\{A\subseteq X:A\text{ is a pseudocompact subset of } X\},$
	 \item $\mathcal{A}_L=\{A\subseteq X:A\text{ is a Lindel\"{o}f subset of }X\}.$
	 
	  	\end{itemize}

	  Also for any $f\in C(X)$, $A\subseteq X$ and $\epsilon>0$, let us denote the set $$B(f,A,\epsilon)=\{g\in C(X):\sup\limits_{x\in A}|f(x)-g(x)|<\epsilon\}.$$ Then the following collections 
	  $\mathcal{B}_i=\{B(f,A,\epsilon):f\in C(X), A\in \mathcal{A}_i\text{ and }\epsilon>0\}$ forms an open base for the $C_i(X)$ where $i=k,kz,ps,L$ \cite{JK 2014,MN 1986,KR 1995}.

\begin{definition}
	The {density} \( d(X) \) of a space \( X \) is defined by
	\[
	d(X) = \aleph_0 + \min\{|D| : D \text{ is a dense subset of } X\}.
	\]
	
	A space \( X \) is {separable} if and only if \( d(X) = \aleph_0 \).
\end{definition}
\begin{definition}
	We call a space $X$ $\aleph_\alpha$-separable if $d(X)<\aleph_\alpha$.
\end{definition}

From the definition it is clear that a space $X$ is separable if and only if it is $\aleph_1$ separable. Also, note that if $X$ is $\aleph_\alpha$-separable then $X$ is $\aleph_\beta$ separable for any cardinal number $\aleph_\beta\geq \aleph_\alpha.$ Since $X$ is a Tychonoff space, the following result is immediate.

\begin{theorem}\label{qu}
	For a space $X$, $C_{\aleph_\alpha}(X)=C_u(X)$ if and only if $X$ is $\aleph_\alpha$-separable.
\end{theorem}

\begin{corollary}
	For a space $X$, $C_u(X)=C_s(X)$ if and only if $X$ is separable.
\end{corollary}
\begin{corollary}\label{qm}
	$C_{\aleph_\alpha}(X)=C_m(X)$ if and only if $X$ is a $\aleph_\alpha$-separable pseudocompact space, here $C_m(X)$ stands for $C(X)$ equipped with the $m$-topology\cite{GJ 1960}.
\end{corollary}
\begin{example}
We take any countably infinite number of objects $\{a_1,a_2,\cdots\}$ outside the closed ordinal space $[0,\omega_{\alpha+2}]$. Let $X$ =$[0,\omega_{\alpha+2}]\oplus\oplus_{i=1}^{\infty}\{a_i\}$. It is easy to check that $B(\underline{0},[0,\omega_{\alpha+2}],1)$ is an open set in $C_k(X)$ which is not open in $C_{\aleph_\alpha}(X)$. On the other hand, $B(\underline{0},\{a_1,a_2,\cdots\},1)$ is open in $C_{\aleph_\alpha}(X)$ which is not open in $C_k(X)$. Thus in general, $C_k(X)$ and $C_{\aleph_\alpha}(X)$ are incomparable with $\aleph_\alpha\geq \aleph_1$.
\end{example}

	\begin{example}
	 \begin{enumerate}
			\item Let $X^*$ be the one point compactification a discrete space $X$ such that $|X|>\aleph_2$. Then $X^*$ is compact but not $\aleph_1$-separable and hence, $$ C_p(X^*)\subsetneq C_s(X^*)\subsetneq C_{\aleph_2}(X^*)\subsetneq C_k(X^*)=C_{ps}(X^*)=C_u(X^*)=C_m(X^*).$$

			\item Since $\mathbb{N}$ is a non pseudocompact separable space hence, we have $$C_p(\mathbb{N})=C_{ps}(\mathbb{N})=C_k(\mathbb{N})\subsetneq C_s(\mathbb{N})=C_L(\mathbb{N})=C_{\aleph_2}(X)=\cdots=C_u(\mathbb{N})\subsetneq C_m(\mathbb{N}).$$
			
			\item Let $X=[0,\omega_1)$. Then $X$ is countably compact but not compact. Also, $X$ is not separable \cite{SS COUNTEREXAMPLE1995}. Therefore, we have $$C_p(X)\subsetneq C_s(X)\subsetneq C_k(X)\subsetneq C_{ps}(X)=C_{\aleph_2}(X)=C_u(X)=C_m(X).$$
			
			\item Consider the space $\Psi=X$ described in \cite[5I]{GJ 1960}. This space is separable, pseudocompact but not countably compact. Then we have $$C_k(X)\subsetneq C_{ps}=C_s(X)=C_{\aleph_2}(X)=\cdots =C_u(X)=C_m(X).$$
		
		\end{enumerate}
	\end{example}

	\section{Connected and compact subset of $C_{\aleph_\alpha}(X)$}
\subsection{Connectedness of $C_{\aleph_\alpha}(X)$}

One of the main results in this section tells that, there is no distinction between the component, the quasicomponent and the path component of any point in $C_{\aleph_\alpha}(X)$. The following fact is helpful to us, towards achieving this result.

\begin{notation}
	For any subset $A$ of $X$, let $A^*(X)=\{f\in C(X):f \text{ is bounded on }A\}$. 
\end{notation}

It is easy to prove that for any $A\in \mathcal{A}_{\aleph_\alpha}$, $A^*(X)$ is a clopen subset of $C_{\aleph_\beta}(X)$ for any cardinal number $\aleph_\beta\geq \aleph_\alpha$ and this is because for any $f\in A^*(X)$, $B(f,A,1)\subseteq A^*(X)$ and for any $g\notin A^*(X)$, $B(g,A,1)\cap A^*(X)=\emptyset$. It is trivial that $C^*(X)\subset A^*(X)$ for any $A\in \mathcal{A}_{\aleph_\alpha}$. On the other hand if $f\in C(X)\setminus C^*(X)$, then there exists a countably infinite subset $D$ of $X$ such that $f$ is unbounded on $D$ and hence $f\notin D^*(X)$. This leads to the following remarks.
\begin{remark}\label{ C^*(X) open C_s}
	For a $\aleph_\alpha$-separable space $X$, $C^*(X)$ is open in $C_{\aleph_\alpha}(X)$.
\end{remark} 
\begin{remark}
	$C^*(X)$ is closed in $C_{\aleph_\alpha}(X)$ for  each $\aleph_\alpha\geq \aleph_1$. Hence, $C^*(X)$ is dense in $C_{\aleph_\alpha}(X)$ if and only if $X$ is pseudocompact.
\end{remark}
\begin{remarks}\label{C^*(X) representation }
	\begin{enumerate}
		\item 	$C^*(X)=\bigcap\limits_{A\in \mathcal{A}_{\aleph_\alpha\text{ and }\aleph_\alpha\geq \aleph_1}}A^*(X)$.
		\item 	$C(X)=\bigcap\limits_{A\in \mathcal{A}_{\aleph_0}}A^*(X)$.
	\end{enumerate}
\end{remarks}
\begin{theorem}
	$\bigcap\limits_{A\in \mathcal{A}_{\aleph_\alpha}}A^*(X)$ is pathwise connected in $C_{\aleph_\alpha}(X)$.
\end{theorem}
\begin{proof}
	We shall show that any arbitrary $h\in \bigcap\limits_{A\in \mathcal{A}_{\aleph_\alpha}}A^*(X)$ is pathwise connected to $\underline{0}$ in $C_{\aleph_\alpha}(X)$. For that purpose, we define the map $p:[0,1]\rightarrow \bigcap\limits_{A\in \mathcal{A}_{\aleph_\alpha}}A^*(X)$ by the rule: $p(t)=th$ for all $t\in [0,1]$.  Once we prove that $p$ is a continuous map, it is clear that $p$ is a path joining $\underline{0}$ and $h$ in $ \bigcap\limits_{A\in \mathcal{A}_{\aleph_\alpha}}A^*(X)$ because $p([0,1])\subseteq \bigcap\limits_{A\in \mathcal{A}_{\aleph_\alpha}}A^*(X)$. Let $B(p(t),A,\epsilon)$ be a basic open neighborhood of the point $p(t)$ in the space $C_{\aleph_\alpha}(X)$, here $A\in \mathcal{A}_{\aleph_\alpha}$, $\epsilon>0$ and $t\in [0,1]$. Now we can write, $|h(x)|\leq M$ for each $x\in A$, for some $M>0$. Let $\delta=\frac{\epsilon}{4M}$, then for each $t'\in [0,1]$ with $|t-t'|<\delta$ and for each $x\in A$, we have $|p(t)(x)-p(t')(x)|<\frac{\epsilon}{2}$, this implies that $\sup\limits_{x\in A}|p(t)(x)-p(t')(x)|<\epsilon$. This proves the continuity of $p$ at the point $t$.
\end{proof}
\begin{corollary}\label{C^*(X) is component C_s}
	$C^*(X)$ is pathwise connected in the space $C_{\aleph_\alpha}(X)$ for all $\aleph_\alpha\geq \aleph_1$.
\end{corollary}

\begin{corollary}
	$C(X)$ is pathwise connected and hence, connected in $C_p(X)$.
\end{corollary}
\begin{corollary}\label{Clopen contained in C*}
	There does not exist any clopen subset $A$ of $C_{\aleph_\alpha}(X)$ contained properly in $C^*(X)$.
\end{corollary}
\begin{corollary}\label{pseudocompact pointopen}
	$C^*(X)$ is clopen in $C_p(X)=C_{\aleph_0}(X)$ if and only if $X$ is pseudocompact.
\end{corollary}
 Corollary \ref{pseudocompact pointopen} may not hold good with $\aleph_\alpha\geq \aleph_1$. 
 \begin{example}
 	Consider $\mathbb{R}$ with the usual topology. Then $C^*(\mathbb{R})=\mathbb{Q}^*(\mathbb{R})$ is clopen in $C_s(\mathbb{R})=C_{\aleph_1}(X)$. But, $\mathbb{R}$ is not pseudocompact.
 \end{example}
\begin{theorem}\label{ disconnection C_s}
	$C^*(X)$ is a maximal connected subset in the space $C_{\aleph_\alpha}(X)$ for all $\aleph_\alpha\geq \aleph_1$.
\end{theorem}
\begin{proof}
	From Corollary \ref{C^*(X) is component C_s}, $C^*(X)$ is a connected subset in the space $C_{\aleph_\alpha}(X)$. Let $S$ be a subset of $C(X)$ such that $C^*(X)\subsetneq S$. Let us choose, $f\in S\setminus C^*(X)$ and $A=\{x_1,x_2,\cdots\}$ be a countably infinite subset of $X$ such that $|f(x_n)|>n$ for each $n\in \mathbb{N}$. It is easy to check that $S=(S\cap A^*(X))\cup (S\cap (C(X)\setminus A^*(X))$ is a disconnection of $S$ in the space $C_{\aleph_\alpha}(X)$.
\end{proof}
\begin{theorem} \label{path,quasi Calpha}
	The path component $P$, the component $C$ and the quasicomponent $Q$ of the function $\underline{0}$ in the space $C_{\aleph_\alpha}(X)$ are all identical to $C^*(X)$ for any $\aleph_\alpha\geq \aleph_1$.
	
\end{theorem}
\begin{proof}
	It follows from Corollary \ref{C^*(X) is component C_s} that $C^*(X)\subseteq P$. It is well known that $P\subseteq C\subseteq Q$. But $Q=$ the intersection of all clopen subset of $C_{\aleph_\alpha}(X)$ containing $\underline{0}$. Hence, from Remark \ref{C^*(X) representation }, $Q\subseteq C^*(X)$ and consequently, $P=C=Q=C^*(X)$.

\end{proof}
\begin{corollary}
		The path component $P$, the component $C$ and the quasicomponent $Q$ of the function ${f}$ in the space $C_{\aleph_\alpha}(X)$ are all identical to $f+C^*(X)$ for any $\aleph_\alpha\geq \aleph_1$.
	
\end{corollary}
For any space $X$, $|C^*(X)|\geq\mathfrak{c}$, which leads to the following remark.
\begin{remark}
	$C_{\aleph_\alpha}(X)$ is never totally disconnected and hence, never extremally disconnected.
\end{remark}

\begin{theorem}
	The following statements are equivalent for a space $X$ and a cardinal number $\aleph_\alpha\geq \aleph_1$:
	\begin{enumerate}
		\item $X$ is pseudocompact.
		\item $C_{\aleph_\alpha}(X)$ is pathwise connected.
		\item $C_{\aleph_\alpha}(X)$ is connected.
		\item $C_{\aleph_\alpha}(X)$ is a topological ring.
	\end{enumerate}
\end{theorem}

\begin{proof}
	Let $X$ be pseudocompact, i.e., $C(X)=C^*(X)$, then it follows from Corollary \ref{C^*(X) is component C_s} that $C_{\aleph_\alpha}(X)$ is pathwise connected and hence, connected. 
	
	Furthermore, the hypothesis that $X$ is pseudocompact implies in a routine manner that $C_{\aleph_\alpha}(X)$ is a topological ring because if $f,g\in C(X)=C^*(X)$, $\epsilon>0$ and $A\in \mathcal{A}_{\aleph_\alpha}$ then $$B(f,A,\delta)\cdot B(g,A,\delta)\subseteq B(f\cdot g,A,\epsilon)$$ if we choose $\delta=\min \{\frac{\epsilon}{4(M+1)},1\}$, where $\sup\limits_{x\in X}|f(x)|<M$ and $\sup\limits_{x\in X}|g(x)|<M$. On the other hand if $X$ is not pseudocompact, then from Theorem \ref*{ disconnection C_s}, $C(X)$ is disconnected. This proves $(3)\implies (1)$. To complete this theorem it remains to show that $(4)\implies (1)$: Let $f\in C(X)$ and $A\in \mathcal{A}_{\aleph_\alpha}$. Then there exist $A_1,A_2\in \mathcal{A}_{\aleph_\alpha}$ and $\epsilon>0$ such that $$B(\underline{0},A_1,\epsilon)\cdot B(f,A_2,\epsilon)\subseteq B(\underline{0},A,1)$$ This implies that $|\frac{\epsilon}{2}f|\leq1$ on $A$. Thus, $f$ is bounded on $A$, in other words $f\in A^*(X)$. Since this relation is true for each $A\in \mathcal{A}_{\aleph_\alpha}$, it follows from Remark \ref{C^*(X) representation } that $f\in C^*(X)$. Hence, $X$ is pseudocompact.
\end{proof}
\begin{corollary}
		The following statements are equivalent for a space $X$:
	\begin{enumerate}
		\item $X$ is pseudocompact.
		\item $C_{s}(X)$ is pathwise connected.
		\item $C_{s}(X)$ is connected.
		\item $C_{s}(X)$ is a topological ring.
	\end{enumerate}
\end{corollary}
Let $A\in \mathcal{A}_{\aleph_\alpha}$, where $\aleph_\alpha\geq \aleph_1$. Define a relation $\sim_A$ on $C(X)$ as follows: $f\sim_A g$ if and only if $\sup\limits_{x\in A}|f(x)-g(x)|<\infty$ for all $f,g\in C(X)$. It is easy to check that this $\sim_A$ is an equivalence relation on $C(X)$. Let $f\in C(X)$, then $[f]_A$ represents the class of $f$, i.e., $[f]_A=\{g\in C(X):f\sim_A g\}$.
\begin{theorem}\label{equi class}
\begin{enumerate}
	\item 	$[\underline{0}]_A=C(X)$ if and only if $A$ is a bounded subset of $X$, i.e., every $f\in C(X)$ is bounded on $A$.
	\item $[f]_A$ is a clopen subset of $C_{\aleph_\alpha}(X)$, for each $f\in C(X)$.
\end{enumerate}
\end{theorem}
\begin{corollary} \label{topologicalsum}
	$C_{\aleph_\alpha}(X)$ is the topological sum of distinct members of the family $\{[f]_A:f\in C(X)\}$.
\end{corollary}
\begin{theorem}
	If $A$ is an  unbounded subset of $X$, then $C_{\aleph_\alpha}(X)$ can be written as the topological sum of at least $\mathfrak{c}$ many subspaces. 
\end{theorem}
\begin{proof}
	In view of Theorem \ref{equi class} and Corollary \ref{topologicalsum}, it is enough to show that there are $\mathfrak{c}$ many disjoint equivalence classes. Since $A$ is not a bounded subset of $X$, we can find an $f\in C(X)$ such that $f$ is unbounded on $A$. Let $k\in [1,2]$ be a real number and we define the function $h_k=kf\in C(X)$. We assert that for $k,s\in [1,2]$ with $k\neq s$ implies that $[h_k]_A\neq [h_s]_A$. Now, it is enough to show that $h_k\notin [h_s]_A$. Choose, an $n\in \mathbb{N}$, then we can always find a $M_n>0$ such that $|k-s|M_n>n$. Since $f$ is unbounded on $A$, there exists a $c_M\in A$ such that $|f(c_M)|>M_n$. Now, $|h_k(c_M)-h_s(c_M)|=|kf(c_M)-sf(c_M)|>M_n|k-s|>n$.
\end{proof}

\begin{corollary}
	If the collection  $\mathcal{A}_{\aleph_\alpha}$($\aleph_\alpha\geq \aleph_1$) contains an unbounded set $A$, then $C_{\aleph_\alpha}(X)$ has at least $\mathfrak{c}$ many clopen sets.
\end{corollary}
In view of Corollary \ref{Clopen contained in C*} and preceding corollary we have the following immediate remark.
\begin{remark}
The cardinality of the set of all non-empty clopen subsets of $C_{\aleph_\alpha}(X)$ is either 1 or at least $\mathfrak{c}$.
\end{remark}
\begin{theorem}
	For any space $X$, it is always possible to write $C_{\aleph_\alpha}(X)(\aleph_\alpha\geq \aleph_1)$ as the topological sum of at least $\mathfrak{c}$ many clopen subsets of $C_{\aleph_\alpha}(X)$ unless $X$ is a pseudocompact space.
\end{theorem}
\begin{corollary}
For a non pseudocompact space $X$, $C_s(X)$ can be expressed as the topological sum of at least $\mathfrak{c}$ many clopen subsets of $C_s(X)$.
\end{corollary}
\subsection{Local compactness of $C_{\aleph_\alpha}(X)$} 	
In this section we show that, like $C_p(X)$, local compactness, hemicompactness and local pseudocompactness are equivalent for the space $C_{\aleph_\alpha}(X)$. From this section onward, we also assume that $\aleph_\alpha\geq \aleph_1$ for the space $C_{\aleph_\alpha}(X)$.
\begin{definition}	
	A topological space $X$ is said to be hemicompact if there exists a countable family of compact subsets $\{K_n:n\in \mathbb{N}\}$ of $X$ with the following property: given any compact subset $K$ of $X$, there exists $n\in \mathbb{N}$ such that $K\subseteq K_n$.
\end{definition}
It is clear that each compact space is hemicompact and every hemicompact space is $\sigma$-compact \cite{E TOPOLOGY1989}.

\begin{lemma}\label{nowhere locallycompact C_s}
	If $X$ is an infinite subset, then for any compact subset $K$ of $C_{\aleph_\alpha}(X)$, interior of $K$ in $C_{\aleph_\alpha}(X)$ is empty, i.e., $C_{\aleph_\alpha}(X)$ is nowhere locally compact.
\end{lemma}
\begin{proof}
	If possible, let the interior of $K$ in $C_{\aleph_\alpha}(X)$ be non-empty. Then there exists an $f\in C_{\aleph_\alpha}(X)$, $A\in \mathcal{A}_{\aleph_\alpha}$ and $\epsilon>0$ such that $B(f,A,\epsilon)\subseteq K$. Since $X$ is an infinite set, without loss of generality we can assume that $A$ is a countably infinite subset of $X$. Let $\{x_n:n\in \mathbb{N}\}$ be a countably infinite subset of $A$ such that $x_m\neq x_n$ if $m\neq n$. Since $K$ is compact, there are finitely many functions $g_1,g_2,\cdots,g_n$ in $K$ such that $K\subseteq \bigcup\limits_{i=1}^n B(g_i,A,\frac{\epsilon}{4})$. Consider the set $\{x_1,x_2,\cdots,x_{n+1}\}$ in $X$. Since $X$ is Tychonoff, for each $i\in\{1,2,\cdots,n+1\}$, there exists a continuous function $p_i:X\rightarrow[-\frac{\epsilon}{2},\frac{\epsilon}{2}]$ such that $p_i(x_i)=\frac{\epsilon}{2}$ and $p_i(x_j)=0$ if $i\neq j$, $j\in \{1,2,\cdots,n+1\}$. It is easy to check that for each $i\in \{1,2,\cdots,n+1\}$, $f+p_i\in B(f,A,\epsilon)$. Consequently, there exist indices $t,m\in \{1,2,\cdots,n+1\}$ such that $t\neq m$, $\{f+p_t, f+p_m\}\subseteq B(g_s,A,\frac{\epsilon}{4})$ and $s\in \{1,2,\cdots,n\}$. This implies for each $x\in A$ that $|(f+p_t)(x)-g_s(x)|<\frac{\epsilon}{4}$ and $|(f+p_m)(x)-g_s(x)|<\frac{\epsilon}{4}$. Consequently, for each $x\in A$, $|(f+p_t)(x)-(f+p_m)(x)|<\frac{\epsilon}{2}$ and therefore, for each $x\in A$, $|p_t(x)-p_m(x)|<\frac{\epsilon}{2}$. In particular with $x=x_t$, this implies $|p_t(x_t)-p_m(x_t)|<\frac{\epsilon}{2}$. But, we note that $p_t(x_t)=\frac{\epsilon}{2}$ and $p_m(x_t)=0$, so we arrive at a contradiction. 
\end{proof}

Before addressing another principal theorem of this section, we reproduce here the following two facts about the space $C_p(X)$ established in details in \cite{T C_p 2011}. 
\begin{lemma} \label{e_x C_s}
	\cite[Problem S067, Page 80]{T C_p 2011}
	
	For each point $x\in X$, the function $e_x:C_p(X)\rightarrow \mathbb{R}$ defined by $e_x(f)=f(x)$ for all $f\in C_p(X)$ is continuous, i.e., $e_x\in C_p(C_p(X))$.
\end{lemma}
\begin{lemma}\label{C_p sigmacompact}
	\cite[Problem S186, Page 152]{T C_p 2011} 
	
	If $C_p(X)$ is $\sigma$-countably compact, then the space $X$ is finite.
	
\end{lemma}
\begin{theorem}\label{variouscompact Calpha}
	The following statements are equivalent for a space $X$.
	\begin{enumerate}
		\item $X$ is finite.
		\item $C_{\aleph_\alpha}(X)$ is locally compact.
		\item $C_{\aleph_\alpha}(X)$ is locally countably compact.
		\item $C_{\aleph_\alpha}(X)$ is locally pseudocompact.
		\item $C_{\aleph_\alpha}(X)$ is hemicompact.
		\item $C_{\aleph_\alpha}(X)$ is $\sigma$-compact.
	\end{enumerate}
\end{theorem}
\begin{proof}
	$(1)\implies (2)$ is trivial because if $|X|=n\in \mathbb{N}$, then $C_{\aleph_\alpha}(X)$ is the same as then space $\mathbb{R}^n$, which is locally compact.
	
	$(2)\implies(3)\implies(4)$: Trivial.

	$(2)\implies (1)$: Follows from the Lemma \ref{nowhere locallycompact C_s}.

	$(4)\implies (2)$: We first show that $X$ is a $\aleph_\alpha$-separable space. We argue by contradiction and assume that $X$ is not $\aleph_\alpha$-separable. Now by $(4)$, the function $\underline{0}$ has a pseudocompact neighborhood $L$ in the space $C_{\aleph_\alpha}(X)$. So, we can write $B(\underline{0},A,\epsilon )\subseteq L$, for some $A\in \mathcal{A}_{\aleph_\alpha}$ and $\epsilon>0$ in $\mathbb{R}$. Since $\overline{A}\subsetneq X$, we can choose a point $y\in X\setminus \overline{A}$. By using the Tychonoffness of the space $X$, we can find out for each $n\in \mathbb{N}$, an $f_n\in C_{\aleph_\alpha}(X)$ such that $f_n(y)=n$ and $f_n(\overline{A})=\{0\}$. It is clear that $f_n\in B(\underline{0},A,\epsilon)$ and hence, $f_n\in L$. We note that the function $e_y:C_p(X)\rightarrow\mathbb{R}$ defined by $e_y(f)=f(y)$ and therefore, $e_y(f_n)=f_n(y)=n$ is continuous by Lemma \ref{e_x C_s}. Since $C_p(X)\subseteq C_{\aleph_\alpha}(X)$, it follows that $e_y:C_{\aleph_\alpha}(X)\rightarrow \mathbb{R}$ is also a continuous map. Thus $e_y\in C_{}(C_{\aleph_\alpha}(X))$ and takes the value $n$ at each point $f_n\in L$ of $C_{\aleph_\alpha}(X)$. Since $n\in \mathbb{N}$ is chosen arbitrarily, it follows that $e_y$ is an unbounded continuous function on the subset $L$ of $C_{\aleph_\alpha}(X)$. This contradicts the pseudocompactness of $L$. Thus $X$ is a $\aleph_\alpha$-separable space. It follows from Theorem \ref{metrizability} to be proved in the next section that $C_{\aleph_\alpha}(X)$ is a metrizable space. Since $C_{\aleph_\alpha}(X)$ is already assumed to be locally pseudocompact, it follows that $C_{\aleph_\alpha}(X)$ is locally compact.
	
	$(1)\implies (5)\implies (6)$: Trivial. 
	
	To complete the theorem we need to proof $(6)\implies (1)$: So, our hypothesis is $C_{\aleph_\alpha}(X)$ is $\sigma$-compact. In particular $C_{\aleph_\alpha}(X)$ is $\sigma$-countably compact. Hence, we can write, $C_{\aleph_\alpha}(X)=\bigcup\limits_{i=1}^\infty K_i$, where $\{K_i:i\in \mathbb{N}\}$ is a countable family of countably compact subsets of $C_{\aleph_\alpha}(X)$. Since $C_p(X)\subseteq C_{\aleph_\alpha}(X)$, it follows that each $K_n$ is countably compact subset of $C_p(X)$. Thus $C_p(X)$ turns out to be a $\sigma$-countably compact space. It follows from Lemma \ref{C_p sigmacompact} that $X$ is a finite set.
\end{proof}

Since, for each $n\in \mathbb{N}$, the constant function $\underline{n}\in C(X)$, where $\underline{n}(x)=n$ for all $x\in X$, for any $y\in X$ then $e_y\in C(C_{\aleph_\alpha}(X))$, which takes the value $n$ at each point $\underline{n}\in C_{\aleph_\alpha}(X)$. This facts leads to the following corollary.
\begin{corollary}
	For any space $X$, $C_{\aleph_\alpha}(X)$ is never pseudocompact.
\end{corollary}

	\section{Metrizability of $C_{\aleph_\alpha}(X)$}
	In this section, we first show that the metrizability of $C_{\aleph_\alpha}(X)$ is equivalent to several well-known topological properties each weaker in general than even the first countability of a space. Subsequently, we realize that a number of familiar completeness properties become equivalent for the space $C_{\aleph_\alpha}(X)$. To make the paper self-contained we reproduce the formal definition of all these properties: A subset $S$ of $Y$ is said to have a countable character if there exists a sequence $\{w_n:n\in \mathbb{N}\}$ of open subsets of $Y$ each containing $S$ such that if $W$ is an open subset of $Y$ containing $S$, then $w_n\subseteq W$ for some $n\in \mathbb{N}$. A space $Y$ is said to be of pointwise countable type if each point of $Y$ is contained in a compact set with countable character. A space $Y$ is said to be a q-space if for each point $y\in Y$, there exists a sequence $\{U_n:n\in \mathbb{N}\}$ of neighborhoods of $y$ such that if $y_n\in U_n$ for each $n\in \mathbb{N}$, then $\{y_n:n\in \mathbb{N}\}$ has a cluster point in $Y$.

	\begin{theorem}\label{metrizability}
		The following statements are equivalent for a space $X$.
		\begin{enumerate}
			\item $X$ is $\aleph_\alpha$-separable.
			\item $C_{\aleph_\alpha}(X)=C_u(X)$.
			\item $C_{\aleph_\alpha}(X)$ is metrizable by a complete metric.
			\item $C_{\aleph_\alpha}(X)$ is first countable.
			\item $C_{\aleph_\alpha}(X)$ is a q-space.
			\item $C_{\aleph_\alpha}(X)$ is of pointwise countable type.
			
		\end{enumerate}
	\end{theorem} 
\begin{proof}
 $(1)\implies(2)$: Let $X$ be $\aleph_\alpha$-separable with a subset $D$ of $X$, dense in it and $|D|< \aleph_\alpha$. Then $B_u(\underline{0},\epsilon)\equiv \{f\in C(X):\sup\limits_{x\in X}|f(x)|<\epsilon\}$, $\epsilon>0$ is a typical basic open neighborhood of $\underline{0}$ in the space $C_u(X)$ and we see that $B(\underline{0},D,\frac{\epsilon}{2})\subseteq B_u(\underline{0},\epsilon)$. This proves that $C_u(X)=C_{\aleph_\alpha}(X)$.
 
 $(2)\implies(3)$ is immediate because $C_u(X)$ is completely metrizable \cite[Theorem 2.6]{MKJ FUNCTIONSPACE2018}.
 
 $(3)\implies (4)\implies(5)$: Trivial.

 Since every metric space is of pointwise countable type and every pointwise countable type space is a q-space (these facts are well-known and in fact could also be proved by routine arguments), the implication relations $(3)\implies(6)\implies(5)$ hold.

 $(5)\implies (1)$: Since $C_{\aleph_\alpha}(X)$ is a $q$-space, there exists a sequence of basic open neighborhoods of $\underline{0}$, $\{B(\underline{0},A_n,\epsilon_n)\}$ where $A_n\in \mathcal{A}_{\aleph_\alpha}$ and $\epsilon_n>0$, such that if we choose $g_n\in B(\underline{0},A_n,\epsilon_n)$ arbitrarily for each $n\in \mathbb{N}$, then the set $\{g_n:n\in \mathbb{N}\}$ has a cluster point in $C_{\aleph_\alpha}(X)$. We claim that $D=\bigcup\limits_{n\in \mathbb{N}}A_n$ is dense in $X$ and hence, $X$ is $\aleph_\alpha$-separable. If possible, let $D$ be not dense in $X$. Then there exists a point $x\in X\setminus \overline{D}$. By using Tychonoffness of $X$, for each $n\in \mathbb{N}$, there exists an $f_n\in C(X)$ such that $f_n(\overline{A_n})=\{0\}$ and $f_n(x)=n$. It is clear that $f_n \in B(\underline{0},A_n,\epsilon_n)$ for each $n\in \mathbb{N}$. Therefore, $\{f_n\}$ has a cluster point in $C_{\aleph_\alpha}(X)$. Hence, the open neighborhood, $B(f,\{x\},1)$ of $f$ in the space $C_{\aleph_\alpha}(X)$ contains an $f_m$ for some $m\in \mathbb{N}$. This implies that $f(x)-1<f_m(x)<f(x)+1$, i.e., $f(x)-1<m<f(x)+1$$\cdots$(i).
 
 Again there exists an $n>4m+4$ such that $f_n\in B(f,\{x\},1)$. Consequently, $f(x)-1<n<f(x)+1$$\cdots$(ii). It follows from (i) and (ii) that $f(x)-1<m<4m+2<n-2<f(x)-1$- which is a contradiction.
 \end{proof}

A space $Y$ is said to be \v{C}ech-complete if it is a $G_\delta$-subset of $\beta Y$, the Stone-\v{C}ech compactification of $Y$. It turns out as the following result suggests that complete metrizability and \v{C}ech-completeness become equivalent for the space $C_{\aleph_\alpha}(X)$.

\begin{theorem}\label{completeness Calpha}
	The following statements are equivalent for the space $C_{\aleph_\alpha}(X)$.
	\begin{enumerate}
		\item $C_{\aleph_\alpha}(X)$ is completely metrizable.
		\item $C_{\aleph_\alpha}(X)$ is a \v{C}ech-complete space.
		\item $C_{\aleph_\alpha}(X)$ is locally \v{C}ech-complete.
		\item $C_{\aleph_\alpha}(X)$ is an open continuous image of a paracompact \v{C}ech-complete space.
		\item $C_{\aleph_\alpha}(X)$ is an open continuous image of a \v{C}ech-complete space.
		\item $C_{\aleph_\alpha}(X)$ is of pointwise countable type.
		\item $X$ is $\aleph_\alpha$-separable.
	\end{enumerate}
\end{theorem} 
\begin{proof}
	Since a completely metrizable space is \v{C}ech-complete \cite[Theorem 4.3.26]{E TOPOLOGY1989}, $(1)\implies (2)$ follows. Also a locally \v{C}ech-complete space is an open continuous image of a \v{C}ech-complete space \cite[3.12.19(d)]{E TOPOLOGY1989}, $(3)\implies(5)$ follows.
	
	$(5)\implies(6)$: Let $C_{\aleph_\alpha}(X)$ is an open continuous image of a   \v{C}ech-complete space $Y$. So, $Y$ is a $G_\delta$-subset of $\beta Y$. Now, every compact space is of pointwise countable type and a $G_\delta$-subspace of a space of pointwise countable type is also of pointwise countable type. It follows that $Y$ is of pointwise countable type. Since the property of being a pointwise countable type space is preserved by an open continuous map, it follows that $C_{\aleph_\alpha}(X)$ is a space of pointwise countable type. Hence, by Theorem \ref{metrizability}, $C_{\aleph_\alpha}(X)$ is completely metrizable. $(2)\implies (3)$ and $(4)\implies(5)$ are trivial.
	
	$(6)\implies (7)\implies(1)$ follows from Theorem \ref{metrizability}.
	
	Now if $(5)$ is true, so also is therefore $(7)$. It follows follows from Theorem \ref{metrizability}, $C_{\aleph_\alpha}(X)$ is a complete metric space and every metric space is paracompact, the statement $(4)$ follows immediately.
\end{proof}
In considering $\aleph_\alpha=\aleph_1$, we establish a direct improvement of Proposition 4.1 in \cite{CRMR 2025}.

\begin{corollary}
		The following statements are equivalent for the space $C_s(X)$.
	\begin{enumerate}
		\item $C_{s}(X)$ is completely metrizable.
		\item $C_{s}(X)$ is a \v{C}ech-complete space.
		\item $C_{s}(X)$ is locally \v{C}ech-complete.
		\item $C_{s}(X)$ is an open continuous image of a paracompact \v{C}ech-complete space.
		\item $C_{s}(X)$ is an open continuous image of a \v{C}ech-complete space.
		\item $C_{s}(X)$ is of pointwise countable type.
		\item $X$ is separable.
	\end{enumerate}

\end{corollary}
\section{Relations between Cardinal functions}

We begin this section by providing the definitions for a few conditions, each of which is weaker than the first countability of a topological space. A space $Y$ is called a Fr\'{e}chet Urysohn space if whenever $A\subseteq Y$ and $y\in \overline{A}$, then there exists a sequence $\{y_n\}$ in $A$ such that $y=\lim y_n$. $Y$ is called a sequential space if whenever a subset $Y_0$ of $Y$ contains the limits of all convergent sequences lying in $Y_0$, then $Y_0$ is closed in $Y$. $Y$ is called countably tight if for each $A\subseteq Y$ and $y\in \overline{A}$, there exists a countable subset $B$ of $A$ such that $y\in \overline{B}$.

The tightness of a point $x$ in the space $X$ is denoted by $t(x,X)$ and defined by $$t(x,X)=\aleph_0+\min\{k:\text{ for all }Y\subseteq X\text{ with }x\in \overline{Y},\text{ there is }A\subseteq Y\text{ with }|A|\leq k\text{ and }x\in \overline{A}\}$$ { Tightness }of $X$ is denoted by $t(X)$ and defined by $$t(X)=\aleph_0+\sup\{t(x,X):x\in X\}$$ It is clear that a space $X$ is countably tight if and only if $t(X)=\aleph_0$.
\begin{theorem}
	If $t(C_{\aleph_\alpha}(X))< \aleph_\alpha$, then $X$ is $\aleph_\alpha$-separable.
\end{theorem}
\begin{proof}
	If possible, let $X$ be not $\aleph_\alpha$-separable. Then for each $A\in \mathcal{A}_{\aleph_\alpha}$, we can pickup a point $p_A\in X\setminus \overline{A}$ and an $f_A\in C(X)$ such that $f_A(A)=\{0\}$ and $f_A(p_A)=1$. Let $\mathcal{F}=\{f_A:A\in A_{\aleph_\alpha}\}$. Then for each $A\in \mathcal{A}_{\aleph_\alpha}$ and $\epsilon>0$, $f_A\in B(\underline{0},A,\epsilon)\cap \mathcal{F}$. This implies that $\underline{0}\in \overline{\mathcal{F}}$. Since $t(C_{\aleph_\alpha}(X))<\aleph_\alpha$, there is a subset $\mathcal{S}$ of $\mathcal{F}$ such that $|\mathcal{S}|<\aleph_\alpha$ and $\underline{0}\in \overline{\mathcal{S}}$. Set $B=\{p_A:f_A\in \mathcal{S}\}$ and hence, $B\in \mathcal{A}_{\aleph_\alpha}$. But, $B(\underline{0},B,\frac{1}{2})\cap S=\emptyset$ - this is a contradiction.
\end{proof}
Setting $\aleph_\alpha=\aleph_1$, we have the following corollary, which recovers one part of Proposition 6.7 in \cite{RS 2021}.
\begin{corollary}
	If $C_s(X)$ is countably tight then $X$ is separable.
\end{corollary}
\begin{corollary}\label{countable tight}
For a space $X$, the following statements are equivalent-
\begin{enumerate}
	\item $X$ is $\aleph_\alpha$-separable.
\item $C_{\aleph_\alpha}(X)=C_u(X)$.
\item $C_{\aleph_\alpha}(X)$ is metrizable by a complete metric.
\item $C_{\aleph_\alpha}(X)$ is first countable.
\item $C_{\aleph_\alpha}(X)$ is a Fr\'echet-Urysohn space. 
\item $C_{\aleph_\alpha}(X)$ is a sequential space.
\item $C_{\aleph_\alpha}(X)$ is countably tight.
	\end{enumerate}
\end{corollary}
\begin{remark}
	Equivalence of the statements (2) through (6) also follow from Theorem 3.14 in \cite{PM 2013}.
\end{remark}
\begin{corollary}
	Theorem 5.4 in \cite{PM 2013} follows immediately from Corollary \ref{countable tight} above with the choice $\aleph_\alpha=\aleph_1$.
\end{corollary}
 Further, we improve Theorem 5.4 in \cite{PM 2013} and Theorem 5.4 in \cite{RS 2021}, see Theorem \ref{C_ all}.

We reproduce the following definitions of a few cardinal functions that we will require in the present section. For any point $x\in X$, $$\chi(x,X)=\aleph_0+\min\{|\mathcal{B}_x|:\mathcal{B}_x\text{ is an open local base for }X\text{ at the point }x\}$$ is called the character of $X$ at the point $x$ and $\chi(X)=\sup\{\chi(x,X):x\in X\}$ is called the character of the space $X$. It is easy to verify that for a topological group $G$, We have $\chi(G)=\chi(a,G)$ for any $a\in G$ and a space $X$ is first countable if and only if $\chi(X)=\aleph_0$.

 Let $\mathcal{V}$ be a collection of non-empty open 
sets in $X$. Then $\mathcal{V}$ is a local $\pi$-base for $x$ if for each open neighborhood 
$R$ of $x$, one has $V\subseteq R$ for some $V\in \mathcal{V}$. If in addition one has $x\in V$ for all 
$V\in \mathcal{V}$, then $\mathcal{V}$ is a local base at $x$. The {$\pi$-character }of $x$ in $X$ is denoted by $\pi\chi(x,X)$ and defined by 
$$\pi\chi(x, X)=\aleph_0+ \min\{|\mathcal{V}|: \mathcal{V}\text{ is a local }\pi\text{-base at }x\} ; $$
The {$\pi$-character} of $X$ is denoted by $\pi(X)$ and defined by $$\pi\chi(X)=\sup\{\pi\chi(x,X):x\in X\}.$$
$X$ is defined to have {countable $\pi$-character }if $\pi\chi(X)=\aleph_0$.

For $x\in X$, $$\psi(x,X)=\aleph_0+\min\{|\mathcal{U}|:\mathcal{U}\text{ is a family of open sets in }X\text{ such that }\bigcap\mathcal{U}=\{x\}\}$$ is called the pseudocharacter of $X$ at the point $x$ and $\psi(X)=\sup\{\psi(x,X):x\in X\}$ is called the pseudocharacter of $X$. A space $X$ is called an $E_0$-space if $\psi(X)=\aleph_0$.

The cardinal number $$c(X)=\aleph_0+\sup\{|\mathcal{U}|:\mathcal{U}\text{ is a family of pairwise disjoint non-empty open sets in }X\}$$ is called the Souslin number of $X$. $X$ is said to have the Souslin property if $c(X)=\aleph_0$. 

The cardinal $$w(X)=\aleph_0+\min\{|\mathcal{B}|:\mathcal{B}\text{ is an open base for }X\}$$ is called the weight of $X$. A space $X$ is second countable if and only if $w(X)=\aleph_0$.

The cardinal $$L(X)=\aleph_0+\min\{k:\text{ each open cover of }X\text{ has a subcover of cardinality}\leq k\} $$ is called the Lindel\"{o}f number of $X$. A space $X$ is Lindel\"{o}f if and only if $L(X)=\aleph_0$.

By $s(X)$ we mean the spread of $X$ and is defined by 
$$s(X)=\aleph_0+\min\{|D|:D\subseteq X\text{ and } D\text{ is discrete}\}$$
If $s(X)=\aleph_0$, we say $X$ has a countable spread.

Additionally, we introduce two new cardinal numbers in this section.
\begin{definition}
	We call a family $\mathcal{B}\subseteq \mathcal{A}_{\aleph_\alpha}$ an almost-$\aleph_\alpha$-cover of $X$ if $\overline{\bigcup\mathcal{B}}=X$. By $ac_{\aleph_\alpha}(X)$ we mean the almost-$\aleph_\alpha$-number of $X$ and define it as follows
	
	$$ac_{\aleph_\alpha}(X)=\aleph_0+\min\{|\mathcal{B}|:\mathcal{B} \text{ is an almost-}\aleph_\alpha\text{-cover of }X\}$$
\end{definition}

Since for any dense subset $D$ of $X$, $\{\{x\}:x\in D\}$ is an almost-$\aleph_\alpha$-cover of $X$, it follows that $ac_{\aleph_\alpha}(X)\leq d(X)$, for any transfinite cardinal number $\aleph_\alpha$. It is easy to verify that for any two transfinite cardinal numbers $\aleph_\alpha$ and $\aleph_\beta$ if $\aleph_\alpha\leq \aleph_\beta$, $ac_{\aleph_\beta}(X)\leq ac_{\aleph_\alpha}(X)$ 

\begin{definition}
	We call a family $\mathcal{B}\subseteq \mathcal{A}_{\aleph_\alpha}$ $\aleph_\alpha$-closed sets in $X$ if for given $A\in \mathcal{A}_{\aleph_\alpha}$, there exists $B\in\mathcal{B}$ such that $A\subseteq \overline{B}$. By $cc_{\aleph_\alpha}(X)$ we mean $\aleph_\alpha$-closed number of $X$ and define it as follows:
	
	$$cc_{\aleph_\alpha}(X)=\aleph_0+\min\{|\mathcal{B}|:\mathcal{B}\text{ is a }\aleph_\alpha\text{-closed sets in }X\}$$
\end{definition}
% Note that,  $cc_{\aleph_1}(X)$ coincides with $\varphi(X)$ introduced in \cite{RS 2021}. 
 \begin{theorem}
 	For any transfinite cardinal number $\aleph_\alpha>\aleph_0$, the following statements are equivalent for a space $X$.
 	\begin{enumerate}
 		\item $X$ is $\aleph_\alpha$-separable.
 		\item $cc_{\aleph_\alpha}(X)=\aleph_0$.
 		\item $ac_{\aleph_\alpha}(X)=\aleph_0$. 
 	\end{enumerate}
 \end{theorem}
\begin{proof}
$(1)\implies (2),(3)$:	If $X$ is $\aleph_\alpha$-separable with a subset $D$ of $X$ such that $\overline{D}=X$ and $|D|< \aleph_\alpha$, then $\{D\}$ is a $\aleph_\alpha$-closed family in $X$ which is also an almost-$\aleph_\alpha$-cover of $X$. Consequently, $cc_{\aleph_\alpha}(X)=\aleph_0=ac_{\aleph_\alpha}(X)$. 

$(2)\implies(1)$: Let $X$ be non $\aleph_\alpha$-separable. Let $\mathcal{B}=\{B_n:n\in \mathbb{N}\}$ be a subfamily of $\aleph_\alpha$-closed sets in $X$. Then for each $n\in \mathbb{N}$, we can choose a point $x_n\in X\setminus\overline{B_n}$. Let $A=\{x_1,x_2,\cdots\}$. Then $A\in \mathcal{A}_{\aleph_\alpha}$ but there does not exist any $n\in \mathbb{N}$ such that $A\subseteq \overline{ B_n}$. Hence, $\mathcal{B}$ cannot be a subfamily of $\aleph_\alpha$-closed sets in $X$ and accordingly $cc_{\aleph_\alpha}(X)>\aleph_0$.

$(3)\implies (1)$:  Let $\mathcal{B}=\{B_n:n\in \mathbb{N}\}$ be an almost-$\aleph_\alpha$-cover of $X$. Then $D=\bigcup \mathcal{B}$ is dense in $X$. Note that, $|D|< \aleph_0 \aleph_\alpha=\aleph_\alpha$. Therefore, $d(X)<\aleph_\alpha$ and hence, $X$ is $\aleph_\alpha$-separable.
\end{proof}

The following theorem gives another representation of density for a space $X$.
\begin{theorem}
For a space $X$, we have $d(X)=ac_{\aleph_1}(X)$.
\end{theorem}
\begin{proof}
	We already observed that $ac_{\aleph_\alpha}(X)\leq d(X)$ for any transfinite cardinal number $\aleph_\alpha$. So, in particular, $ac_{\aleph_1}(X)\leq d(X)$. Let $\mathcal{B}$ be an almost-$\aleph_1$-cover of $X$ such that $|\mathcal{B}|=ac_{\aleph_1}(X)$. Then $D=\bigcup \mathcal{B}$ is dense in $X$. Now $|D|\leq ac_{\aleph_1}(X)$ and hence, $d(X)\leq ac_{\aleph_1}(X)$.
\end{proof}
\begin{remark}
	For a space $X$, we have $d(X)=ac_{\aleph_1}(X)\geq ac_{\aleph_2}(X)\geq \cdots $.
\end{remark}
\begin{example}
	Let $X$ be a discrete space with $|X|=\aleph_1$. Then $ac_{\aleph_1}(X)=d(X)=\aleph_1$. But, $ac_{\aleph_2}(X)=ac_{\aleph_3}(X)=\cdots=\aleph_0$.
\end{example}
\begin{theorem}\label{cel X}
	For a space $X$, $$c(X)\leq ac_{\aleph_\alpha}(X)\aleph_\alpha.$$
\end{theorem}
\begin{proof}
	Let $\mathcal{A}$ be an almost-$\aleph_\alpha$-cover of $X$ such that $|\mathcal{A}|+\aleph_0=ac_{\aleph_\alpha}(X)$ and $\mathcal{B}$ be a cellular family of open sets in $X$. Enough to show that $|\mathcal{B}|\leq |\mathcal{A}| \aleph_\alpha$. For each $A\in \mathcal{A}$, consider the set $\mathcal{M}_A=\{B\in \mathcal{B}:B\cap A\neq \emptyset\}$. It is easy to check that $\mathcal{B}=\bigcup\limits_{A\in \mathcal{A}} \mathcal{M}_A$. Also note that, $\mathcal{B}$ is a cellular family of open sets and hence $|\mathcal{M}_A|\leq |A|<\aleph_\alpha$. Consequently, $|\mathcal{B}|\leq |\mathcal{A}|\aleph_\alpha$.
\end{proof}

\begin{theorem}\label{pseudocharacter Calpha}
	For any space $X$, $\psi(C_{\aleph_\alpha}(X))=ac_{\aleph_\alpha}(X).$
\end{theorem}
\begin{proof}

	Let $\mathcal{B}=\{B(\underline{0},A_\lambda,\epsilon_\lambda):A_\lambda\in \mathcal{A}_{\aleph_\alpha},\epsilon_\lambda>0 \text{ and }\lambda\in \Lambda\}$ be a family of basic open sets in $C_{\aleph_\alpha}(X)$ such that $\bigcap\mathcal{B}=\{\underline{0}\}$ and $|\Lambda|=\psi(C_{\aleph_\alpha}(X))$. We assert that $\{A_\lambda:\lambda\in \Lambda\}$ is an almost-$\aleph_\alpha$-cover of $X$. If possible, let there exist a point $x\in X\setminus \overline{\bigcup \mathcal{B}}$. Using the Tychonoffness of $X$, we can find out an $f\in C(X)$ such that $f(x)=1$ and $f(\overline{\bigcup \mathcal{B}})=\{0\}$. It is clear that $f\in \bigcap\mathcal{B}$ but $f\neq \underline{0}$, a contradiction. Consequently, $\{A_\lambda:\lambda\in \Lambda\}$ is an almost-$\aleph_\alpha$-cover of $X$ and hence, $ac_{\aleph_\alpha}(X)\leq \psi(C_{\aleph_\alpha}(X))$.

	Let $\mathcal{B}$ be an almost-$\aleph_\alpha$-cover of $X$ such that $|\mathcal{B}|+\aleph_0=ac_{\aleph_\alpha}(X)$. Enough to show that $\{\underline{0}\}=D\equiv\bigcap\{B(\underline{0},A,
\frac{1}{n}):n\in \mathbb{N} \text{ and }A\in \mathcal{B}\}$. If possible, let $f\in D$ such that $f\neq \underline{0}$. Then there exists an $x\in X$ such that $f(x)\neq 0$. Without loss of generality, we can assume $f(x)>0$. Choose an $\epsilon>0$ such that $f(x)-\epsilon>0$ and an $n\in \mathbb{N}$ such that $\frac{1}{n}<f(x)-\epsilon$. Since $f$ is continuous at $x$, we can find an open set $U$ containing $x$ in $X$ such that $f(U)\subseteq (f(x)-\epsilon,f(x)+\epsilon)$. Since $\mathcal{B}$ is an almost-$\aleph_\alpha$-cover of $X$, this implies that $\bigcup \mathcal{B}\cap U\neq \emptyset$. Choose a point $q\in \bigcup \mathcal{B}\cap U$. Since $q\in \bigcup \mathcal{B}$, we can find an $A\in \mathcal{B}$ such that $q\in A$. Therefore, $f(q)<\frac{1}{n+2}$ as $f\in D\subseteq B(\underline{0}, A,\frac{1}{n+2})$. But, $q\in U$ implies that $\frac{1}{n}<f(x)-\epsilon<f(q)$ - which is a contradiction. Consequently, $\{\underline{0}\}=D$ and hence, $\psi(C_{\aleph_\alpha}(X))\leq ac_{\aleph_\alpha}(X)$.
\end{proof}

\begin{corollary}
	For a space $X$, we have $\psi(C_s(X))=ac_1(X)=d(X)$.
\end{corollary}
\begin{corollary} \label{E_0 }
	For a space $X$, the following statements are equivalent.
	\begin{enumerate}
		\item $C_{\aleph_\alpha}(X)$ is metrizable.
		\item $C_{\aleph_\alpha}(X)$ is submetrizable (a topological space $(Y,\tau)$ is called submetrizable if one can define a metric $d$ on $Y$ such that the topology induced by $d$ on $Y$ is weaker than the topology $\tau$ on it. It is easy to check that submetrizable spaces are $E_0$).
		\item $C_{\aleph_\alpha}(X)$ is an $E_0$-space.
		\item $X$ is $\aleph_\alpha$-separable.
	\end{enumerate}
\end{corollary}
 Thus the next theorem is an immediate consequence of Theorem \ref{metrizability}, Corollary \ref{countable tight} and Corollary \ref{E_0 }. This result extends the scope of Theorem 5.4 in \cite{RS 2021}.

\begin{theorem}\label{C_ all}
	For a space $X$, the following statements are equivalent-
	\begin{enumerate}
		\item $X$ is separable.
		\item $C_s(X)=C_u(X)$.
		\item $C_s(X)$ is metrizable.
		\item $C_{s}(X)$ is first countable.
		\item $C_{s}(X)$ is a Fr\'echet-Urysohn space. 
		\item $C_{s}(X)$ is a sequential space.
		\item $C_{s}(X)$ is countably tight.
		\item $C_s(X)$ is a q-space.
		\item $C_s(X)$ is of pointwise countable type.
		\item $C_s(X)$ is submetrizable.
	\item $C_s(X)$ is an $E_0$-space.

	\end{enumerate}\end{theorem}

\begin{lemma}
	For a space $X$ and a transfinite cardinal number $\aleph_\alpha$, the collection $\mathcal{C}_{\aleph_\alpha}=\{V[A,\epsilon]:A\in \mathcal{A}_{\aleph_\alpha}\text{ and }\epsilon>0\text{ a real number}\}$ forms a base for some uniformity on $C(X)$, where $V[A,\epsilon]=\{(f,g):|f(x)-g(x)|<\epsilon\text{ for all }x\in A\}$. Furthermore, the topology $\tau(\mathcal{C}_{\aleph_\alpha})$induced from this uniformity coincides with the topology of $C_{\aleph_\alpha}(X)$.
\end{lemma}
\begin{definition}
	
	For a Tychonoff space $X$, the uniform weight is the least infinite cardinality of some compatible uniformity base on $X$. The uniform weight is denoted by $uw(X)$.

\end{definition}
\begin{theorem}\label{character Calpha}
	For a space $X$, $$\pi\chi(C_{\aleph_\alpha}(X))=\chi(C_{\aleph_\alpha}(X))=uw(C_{\aleph_\alpha}(X))=cc_{\aleph_\alpha}(X).$$
\end{theorem}

\begin{proof}
	
	The inequality $\pi\chi(Y)\leq \chi(Y)\leq uw(Y)$ holds for any Tychonoff space $Y$. Moreover, Like any topological group, $\pi\chi(C_{\aleph_\alpha}(X))=\chi(C_{\aleph_\alpha}(X))$\cite[Theorem 3.6]{C HANDBOOK1984}.

		Let $\mathcal{B}$ be an $\aleph_\alpha$-closed family of sets in $X$ such that $|\mathcal{B}|+\aleph_0=cc_{\aleph_\alpha}(X)$. It is easy to verify that $\mathcal{F}=\{V(B,\frac{1}{n}):B\in \mathcal{B}\text{ and }n\in \mathbb{N}\}$ is a compatible uniformity base for $C_{\aleph_\alpha}(X)$ and hence, $uw(C_{\aleph_\alpha}(X))\leq cc_{\aleph_\alpha}(X)$. It suffices to show that $cc_{\aleph_\alpha}(X)\leq \chi(C_{\aleph_\alpha}(X))$. Let $\mathcal{B}=\{B(\underline{0},F_\alpha,\epsilon_\alpha):\epsilon_\alpha>0,F_\alpha\in\mathcal{A}_{\aleph_\alpha},\alpha\in \Lambda\}$ be a local base at $\underline{0}$ in $C_{\aleph_\alpha}(X)$ with $|\mathcal{B}|=\chi(C_{\aleph_\alpha}(X))$. Then it is not hard to prove that $\mathcal{F}=\{F_\alpha:\alpha\in \Lambda\}$ is an $\aleph_\alpha$-closed family of sets in $X$, indeed for any $A\in \mathcal{A}_{\aleph_\alpha}$, there exists an $\alpha\in \Lambda$ such that $B(\underline{0},F_\alpha,\epsilon_\alpha)\subseteq B(\underline{0},A,1)$. We assert that $A\subseteq \overline{F_\alpha}$, for if this is not true, then there exists a point $x\in A$ such that $x\notin \overline{F_\alpha}$. Using the Tychonoffness of $X$, one can produce an $f\in C(X)$ such that $f(\overline{F_\alpha})=\{0\}$ but $f(x)=2$. It follows that $f\in B(\underline{0},F_\alpha,\epsilon_\alpha)$ but $f\notin B(\underline{0},A,1)$ - a contradiction. Hence, $cc_{\aleph_\alpha}(X)\leq \chi(C_{\aleph_\alpha}(X))$.

\end{proof}
As we know a $T_2$ uniform space is metrizable if and only if its uniformity has a countable base. The previous theorem gives an alternative proof of the following corollary.
\begin{corollary}
	For any transfinite cardinal number $\aleph_\alpha$, $C_{\aleph_\alpha}(X)$ is metrizable if and only if $C_{\aleph_\alpha}(X)$ has countable $\pi$-character if and only if $X$ is $\aleph_\alpha$-separable.
\end{corollary}
With the choice $\aleph_\alpha=\aleph_1$, we have an improvement of Proposition 5.7 (2) in \cite{PM 2013}.

\begin{corollary}
	For a space $X$, we have $\pi \chi(C_s(X))=\chi(C_s(X))=uw(C_s(X))=cc_{\aleph_1}(X)=\varphi(X)$.
\end{corollary}
Since for any $T_1$ space $X$,  $\psi(X)\leq \chi(X)$ \cite[3.1.F(a)]{E TOPOLOGY1989} holds. This fact leads us to the following corollaries.
\begin{corollary}\label{compare cc ac}
	For any transfinite cardinal number $\aleph_\alpha$ and for a space $X$, we have $ac_{\aleph_\alpha}(X)\leq cc_{\aleph_\alpha}(X).$
\end{corollary}
\begin{corollary}
	For a space $X$, we have $d(X)=ac_1(X)\leq cc_1(X)$.
\end{corollary}

The next corollary follows immediately from Theorem \ref{cel X} and Corollary \ref{compare cc ac}.
\begin{corollary}
	For a space $X$, we have $$c(C_{\aleph_\alpha}(X))\leq cc_{\aleph_\alpha}(X)\aleph_\alpha.$$
\end{corollary}
\begin{theorem}
	For a space $X$, we have $$d(X)\leq cc_{\aleph_\alpha}(X)\aleph_\alpha.$$
\end{theorem}
\begin{proof}
	Let $\mathcal{A}$ be a $\aleph_\alpha$-closed family of sets in $X$ such that $|\mathcal{A}|+\aleph_0=cc_{\aleph_\alpha}(X)$. It is easy to check that $B=\bigcup\limits_{A\in \mathcal{A}} A$ is dense in $X$ and hence, $d(X)\leq cc_{\aleph_\alpha}(X)\aleph_\alpha$.
\end{proof}
In the proof of the preceding theorem if $cc_{\aleph_\alpha}(X)<\aleph_\alpha$, then it is immediate to observe that $|B|<\aleph_\alpha$. Therefore, we have the following corollary.
\begin{corollary}
 $cc_{\aleph_\alpha}(X)< \aleph_\alpha$ if and only if $X$ is $\aleph_\alpha$-separable. 
\end{corollary}
\begin{corollary}
	$cc_{\aleph_\alpha}(X)=\aleph_\alpha$ then $d(X)=\aleph_\alpha$.
\end{corollary}
\begin{corollary}
	If $cc_{\aleph_\alpha}(X)\geq \aleph_\alpha$ then $d(X)\leq cc_{\aleph_\alpha}(X)$.
\end{corollary}
\begin{theorem}
	For any space $X$, $$w(C_{\aleph_\alpha}(X))=cc_{\aleph_\alpha}(X)d(C_{\aleph_\alpha}(X)).$$
\end{theorem}
\begin{proof}
	For any space $Y$, $d(Y)\leq w(Y)  $ and $\chi(Y)\leq w(Y)$. Therefore, it follows from Theorem \ref{character Calpha} that $cc_{\aleph_\alpha}(X) d(C_{\aleph_\alpha}(X))\leq w(C_{\aleph_\alpha}(X))$. To prove the reverse inequality, let $\mathcal{B}$ be an $\aleph_\alpha$-closed family of sets in $X$ with $|\mathcal{B}|+\aleph_0=cc_{\aleph_\alpha}(X)$ and let $D$ be a dense subset of $C_{\aleph_\alpha}(X)$ such that $|D|=d(C_{\aleph_\alpha}(X))$. Suppose $\mathcal{F}=\{B(f,F,\frac{1}{n}):f\in D,F\in \mathcal{B}\text{ and }n\in \mathbb{N}\}$. We claim that $\mathcal{F}$ is an open base for the space $C_{\aleph_\alpha}(X)$. Indeed, for $h\in C(X)$, $A\in \mathcal{A}_{\aleph_\alpha}$ and $\epsilon>0$ we can select an $n\in \mathbb{N}$ with $\frac{1}{n}<\frac{\epsilon}{4}$ and $F\in \mathcal{B}$ such that $A\subseteq \overline{F}$ and a $g\in D\cap B(h,F,\frac{1}{n})$. It is not hard to check that $h\in B(g,F,\frac{1}{n})\subseteq B(h,A,\epsilon)$. Therefore, $w(C_{\aleph_\alpha}(X))\leq|\mathcal{F}|\leq |D||\mathcal{B}|\aleph_0=cc_{\aleph_\alpha}(X)d(C_{\aleph_\alpha}(X))$.
\end{proof}

\begin{theorem}
	For any space $X$, $$d(C_{\aleph_\alpha}(X))\leq cc_{\aleph_\alpha}(X) L(C_{\aleph_\alpha}(X)).$$
\end{theorem}
\begin{proof}
	Let $\mathcal{B}$ be an $\aleph_\alpha$-closed family of sets in $X$ with $|\mathcal{B}|+\aleph_0=cc_{\aleph_\alpha}(X)$. Let $B\in \mathcal{B}$ and $m\in \mathbb{N}$, then $\{B(f,B,\frac{1}{m}):f\in C(X)\}$ is an open cover of $C_{\aleph_\alpha}(X)$. Therefore, there exists a subcover $\{B(f,B,\frac{1}{m}):f\in C_{B,m}\}$ for $C_{\aleph_\alpha}(X)$ with $C_{B,m}\subseteq C(X)$ with $|C_{B,m}|=L(C_{\aleph_\alpha}(X))$. Let D=$\bigcup\limits_{B\in \mathcal{B}}\bigcup\limits_{m\in \mathbb{N}}C_{B,m}$. We show that $D$ is dense in $C_{\aleph_\alpha}(X)$. Let $B(h,A,\epsilon)$ be a  basic open set in $C_{\aleph_\alpha}(X)$ where $\epsilon>0$, $A\in \mathcal{A}_{\aleph_\alpha}$ and $h\in C(X)$. Then there exists $B\in \mathcal{B}$ such that $A\subseteq \overline{B}$ and there exists $m\in \mathbb{N}$ such that $\frac{1}{m}<\frac{\epsilon}{4}$. Now, there exists $f\in C_{B,m}$ such that $h\in B(f,B,\frac{1}{m})$, this implies that $f\in B(h,B,\frac{1}{m})\subseteq B(h,\overline{B},\epsilon)\subseteq B(h,A,\epsilon)$ and note that $f\in D$. Observe that the denseness of $D$ in $C_{\aleph_\alpha}(X)$ implies that $d(C_s(X))\leq |D|\leq |\mathcal{B}|\aleph_0 L(C_s(X))=cc_{\aleph_\alpha}(X) L(C_{\aleph_\alpha}(X))$.
\end{proof}
\begin{theorem}\label{5.5}
	For a space $X$, $$d(C_{\aleph_\alpha}(X))\leq cc_{\aleph_\alpha}(X)c(C_{\aleph_\alpha}(X)).$$
\end{theorem}
\begin{proof}
	Let $\mathcal{B}$ be an $\aleph_\alpha$-closed family of sets in $X$ such that $|\mathcal{B}|+\aleph_0= cc_{\aleph_\alpha}(X)$. For each $A\in \mathcal{B}$ and $k\in \mathbb{N}$, there exists a cellular family $\mathcal{F}_A^k$ of basic open sets in the space $C_{\aleph_\alpha}(X)$ of the type $B(f,A,\frac{1}{k})$ - indeed we can simply take $\mathcal{F}_A^k=\{B(f,A,\frac{1}{k})\}$, for any $f\in C(X)$. A straight forward use of Zorn's Lemma ensures that there exists for any $A\in \mathcal{B}$ and any $k\in \mathbb{N}$, a maximal cellular family $\gamma_A^k$ of basic open sets of the form $B(f,A,\frac{1}{k})$, $f\in C(X)$. Let $D=\{f\in C(X):\text{ there exist }A\in \mathcal{B}\text{ and }k\in \mathbb{N}\text{ such that }B(f,A,\frac{1}{k})\in \gamma_A^k\}$. It is clear that $|D|\leq cc_{\aleph_\alpha}(X)c(C_{\aleph_\alpha}(X))$. So, to complete this theorem we shall show that $D$ is dense in $C_{\aleph_\alpha}(X)$. We choose $B\in \mathcal{A}_{\aleph_\alpha}$, $\epsilon>0$ and $g\in C(X)$. Since $\mathcal{B}$ is an $\aleph_\alpha$-closed family of sets in $X$, there exists $F\in \mathcal{B}$ such that $B\subseteq \overline{F}$ and we can choose a $m\in \mathbb{N}$ $\frac{1}{m}<\frac{\epsilon}{4}$.  Now either $B(g,F,\frac{1}{m})\in \gamma_F^m$ or $B(g,F,\frac{1}{m})\notin \gamma_F^m$. In the second case, due to the maximality of $\gamma_F^m$, there exists $B(p,F,\frac{1}{m})\in \gamma_F^m$ such that $B(p,F,\frac{1}{m})\cap B(g,F,\frac{1}{m})\neq \emptyset$. Thus, in any case, we can write $B(p,F,\frac{1}{m})\cap B(g,F,\frac{1}{m})\neq \emptyset$ for some $B(p,F,\frac{1}{m})\in \gamma_F^m$. We choose $j\in B(p,F,\frac{1}{m})\cap B(g,F,\frac{1}{m}$). Then $|p-g|\leq|p-j|+|j-g|<\frac{\epsilon}{2}$ and hence, $p\in B(g,B,\epsilon)\cap D$.
	
\end{proof}

Since for any space $Y$, $w(Y)\geq \chi(Y)$ and $w(Y)\geq L(Y)$ \cite[Theorem 2.1]{j CARDINAL1980}, the following cardinal inequality is a direct consequence of Theorem \ref{character Calpha}.

\begin{theorem}
	For a space $X$, $$w(C_{\aleph_\alpha}(X))\geq cc_{\aleph_\alpha}(X) L(C_{\aleph_\alpha}(X)).$$
\end{theorem}

As for any space $Y$, $c(Y)\leq d(Y)$ and $c(Y)\leq s(Y)\leq w(Y)$, we get the comprehensive result as a consequence of the previously discussed theorems.

\begin{theorem}\label{all cardinal calpha}
	For any space $X$, we have $$w(C_{\aleph_\alpha}(X))=cc_{\aleph_\alpha}(X)d(C_{\aleph_\alpha}(X))=cc_{\aleph_\alpha}(X)L(C_{\aleph_\alpha}(X))=cc_{\aleph_\alpha}(X)s(C_{\aleph_\alpha}(X))$$$$=cc_{\aleph_\alpha}(X)c(C_{\aleph_\alpha}(X)).$$
\end{theorem}
The next theorem will be useful towards giving a characterization of spaces $X$ for which $C_{\aleph_\alpha}(X)$ is separable.

\begin{theorem}
	For a space $X$ and a transfinite cardinal number $\aleph_\alpha$, $d(C_{\aleph_\alpha}(X))<\aleph_\alpha$ if and only if $d(C_u(X))<\aleph_\alpha$.
\end{theorem}

\begin{proof}
	It is easy to see that $d(C_{\aleph_\alpha}(X))\leq d(C_u(X))$. So, one part is trivial. Let, $d(C_{\aleph_\alpha}(X))<\aleph_\alpha$ and $D$ be a dense set in $C_{\aleph_\alpha}(X)$ such that $|D|=d(C_{\aleph_\alpha}(X))$. We claim that $D$ is dense in $C_u(X)$. If possible, let $D$ be not dense in $C_u(X)$. So, there exists $f\in C(X)$ and $\epsilon>0$ such that $B(f,\epsilon)\cap D=\emptyset$. Therefore, for each $p\in D$ there exists a point $x_p\in X$ such that $|f(x_p)-p(x_p)|\geq\epsilon$. Let $A=\{x_p:p\in D\}$, then $A\in \mathcal{A}_{\aleph_\alpha}$. It is easy to check that $B(f,A,\frac{\epsilon}{2})\cap D=\emptyset$ - which is a contradiction.
\end{proof}
\begin{corollary}\label{Cu separability}
	For a space $X$, the following statements are equivalent.
	\begin{enumerate}
		\item $C_u(X)$ is separable.
		\item $C_{\aleph_\alpha}(X)$ is separable for any transfinite cardinal $\aleph_\alpha$.
	\end{enumerate}

\end{corollary}
It is well-known that $C_u(X)$ is separable if and only if $X$ is compact and metrizable \cite[see at the end of Page 112]{MHHM 1998}. This fact, together with Theorem \ref{all cardinal calpha} and Corollary \ref{Cu separability} yield the following theorem.
\begin{theorem}\label{all equivCalpha}
	The following statements are equivalent for a space $X$:
	\begin{enumerate}
		\item $C_{\aleph_\alpha}(X)$ is second countable.
		\item $C_{\aleph_\alpha}(X)$ is separable. \item $C_{\aleph_\alpha}(X)$ is Lindel\"{o}f and $X$ is $\aleph_\alpha$-separable.
		\item $C_{\aleph_\alpha}(X)$ has the Souslin property and $X$ is $\aleph_\alpha$-separable.
		\item $C_{\aleph_\alpha}(X)$ has a countable spread and $X$ is $\aleph_\alpha$-separable.
		\item $C_u(X)$ is separable.
		\item $X$ is compact and metrizable.
	\end{enumerate}
\end{theorem}
\begin{proof}
	We only need to prove $(6)\implies (1)$: Indeed, if $X$ is compact and metrizable, then it is separable and hence, $C_{\aleph_\alpha}(X)=C_u(X)$. Therefore, $C_{\aleph_\alpha}(X)$ becomes second countable as it is now a metrizable space.
\end{proof}
\begin{remark}
	Theorem 5.6 in \cite{PM 2013} is a special case of the above theorem on choosing $\aleph_\alpha=\aleph_1$. Also, a special case of the equivalence of (4) and (7) in the above theorem is precisely the Theorem 7.3 in \cite{RS 2021}. 
\end{remark}
	
\end{document}